\documentclass[12pt,reqno]{amsart}
\usepackage{eurosym}
\usepackage{amsfonts}
\usepackage{amsmath}
\usepackage[bookmarksnumbered, colorlinks, plainpages]{hyperref}
\usepackage{graphicx}
\usepackage{epsfig}
\usepackage{epstopdf}
\usepackage{color}
\usepackage{caption}
\usepackage{subcaption}

\hypersetup{colorlinks=true,linkcolor=red, anchorcolor=green, citecolor=cyan, urlcolor=red, filecolor=magenta, pdftoolbar=true}
\newtheorem{theorem}{Theorem}[section]

\newtheorem{proposition}[theorem]{Proposition}
\newtheorem{corollary}[theorem]{Corollary}
\theoremstyle{definition}
\newtheorem{definition}[theorem]{Definition}
\newtheorem{example}[theorem]{Example}

\theoremstyle{remark}
\newtheorem{remark}[theorem]{Remark}
\numberwithin{equation}{section}

\begin{document}
\date{{\scriptsize Received: , Accepted: .}}
\title[Pata Zamfirescu Type Fixed-Disc Results]{Pata Zamfirescu Type
Fixed-Disc Results with a Proximal Application}

\begin{abstract}
This paper is concerning to the geometric study of fixed points of a
self-mapping on a metric space. We establish new generalized contractive
conditions which ensure that a self-mapping has a fixed disc or a fixed circle.
We introduce the notion of a best proximity circle and explore some proximal
contractions for a non-self-mapping as an application. Necessary
illustrative examples are presented to highlight the importance of the
obtained results.
\end{abstract}

\author[N. \"{O}ZG\"{U}R]{N\.{I}HAL \"{O}ZG\"{U}R}
\address[Nihal \"{O}zg\"{u}r]{Bal\i kesir University, Department of
Mathematics, 10145 Bal\i kesir, TURKEY}
\email{nihal@balikesir.edu.tr}
\author[N. TA\c{S}]{N\.{I}HAL TA\c{S}}
\address[Nihal Ta\c{s}]{Bal\i kesir University, Department of Mathematics,
10145 Bal\i kesir, TURKEY}
\email{nihaltas@balikesir.edu.tr}
\maketitle

%\thanks{Bal\i kesir University, Department of Mathematics, 10145 Bal\i
%kesir, TURKEY, nihaltas@balikesir.edu.tr, nihal@balikesir.edu.tr}

\setcounter{page}{1}

%\dedicatory{This paper is dedicated to Professor ABCD}

\section{Introduction and Motivation}

\label{sec:intro}Fixed-point theory has an important role due to solutions
of the equation $Tx=x$ where $T$ is a self-mapping on a metric (resp. some
generalized metric) space. This theory has been extensively studied with
some applications in diverse research areas such as integral equations,
differential equations, engineering, statistics, economics etc. Some
questions have been arisen for the existence and uniqueness of fixed points.
Some fixed-point problems are as follows:

(1) Is there always a solution of the equation $Tx=x$?

(2) What are the existence conditions for a fixed point of a self-mapping?

(3) What are the uniqueness conditions if there is a fixed point of a
self-mapping?

(4) Can the number of fixed points be more than one?

(5) If the number of fixed points is more than one, is there a geometric
interpretation of these points?

Considering the above questions, many researchers have been studied on
fixed-point theory with different aspects.

Some generalized contractive conditions have been investigated to guarantee
the existence and uniqueness of a fixed point of a self-mapping. For
example, in \cite{Zamfirescu}, an existence theorem was given for a
generalized contraction mapping. In \cite{Pata}, a refinement of the
classical Banach contraction principle was obtained. A new generalization of
these results was derived by using both of the above contraction conditions
in \cite{Jacob}.

In the cases in which the fixed-point equation $Tx=x$ has no solution, the
notion of \textquotedblleft \textit{best proximity point}\textquotedblright\
has been appeared as an approximate solution $x$ such that the error $%
d\left( x,Tx\right) $ is minimum. For example, the existence of best
proximity point was investigated using the Pata type proximal mappings in
\cite{Jacob}. These results are the generalizations of ones obtained in \cite%
{Pata}.

If a fixed point is not unique then the geometry of the fixed points of a
self-mapping is an attractive problem. For this purpose, a recent approach
called \textquotedblleft \emph{fixed-circle problem}\textquotedblright\
(resp. \textquotedblleft \emph{fixed-disc problem}\textquotedblright ) has
been studied by various techniques (see \cite{Mlaiki}, \cite{Mlaiki-Axioms},
\cite{Ozgur-Tas-malaysian}, \cite{Ozgur-Tas-circle-thesis}, \cite%
{Ozgur-Tas-Celik}, \cite{ozgur-aip}, \cite{Ozgur-simulation}, \cite%
{Pant-Ozgur-Tas}, \cite{Tas math}, \cite{Tas}, \cite{Tas-fbed}). For
example, in \cite{Ozgur-simulation}, some fixed-disc results have been
obtained using the set of simulation functions on a metric space.

In this paper, mainly, we focus on the geometric study of fixed points of a
self-mapping on a metric space. New generalized contractive conditions are
established for a self-mapping to have a fixed disc or a fixed circle with
some illustrative examples. As an application, we introduce the notion of a
best proximity circle and explore some proximal contractions for a
non-self-mapping.

\section{Main Results}

\label{sec:1} Throughout the section, we assume that $(X,d)$ is a metric
space, $T:X\rightarrow X$ is a self-mapping and $D[x_{0},r]$ is a disc
defined as%
\begin{equation*}
D[x_{0},r]=\left\{ u\in X:d(u,x_{0})\leq r\right\} \text{.}
\end{equation*}

If the self-mapping $T$ fixes all of the points in the disc $D[x_{0},r]$,
that is, $Tu=u$ for all $u\in D[x_{0},r]$, then $D[x_{0},r]$ is called as
the fixed disc of $T$.

To obtain new fixed-disc results, we modify the notion of a Zamfirescu
mapping on metric spaces (see \cite{Zamfirescu} for more details).

\begin{definition}
\label{def1} The self-mapping $T$ is called a Zamfirescu type $x_{0}$%
-mapping if there exist $x_{0}\in X$ and $a,b\in \lbrack 0,1)$ such that%
\begin{equation*}
d(Tu,u)>0\Longrightarrow d(Tu,u)\leq \max \left\{ ad(u,x_{0}),\frac{b}{2}%
\left[ d(Tx_{0},u)+d(Tu,x_{0})\right] \right\} \text{,}
\end{equation*}%
for all $u\in X$.
\end{definition}

\begin{proposition}
\label{prop1} If $T$ is a Zamfirescu type $x_{0}$-mapping with $x_{0}\in X$
then we have $Tx_{0}=x_{0}$.
\end{proposition}

\begin{proof}
Let $T$ be a Zamfirescu type $x_{0}$-mapping with $x_{0}\in X$. Assume that $%
Tx_{0}\neq x_{0}$. Then we have $d(Tx_{0},x_{0})>0$ and using the Zamfirescu
type $x_{0}$-mapping hypothesis, we get%
\begin{eqnarray*}
d(Tx_{0},x_{0}) &\leq &\max \left\{ ad(x_{0},x_{0}),\frac{b}{2}\left[
d(Tx_{0},x_{0})+d(Tx_{0},x_{0})\right] \right\}  \\
&=&\max \left\{ 0,bd(Tx_{0},x_{0})\right\} =bd(Tx_{0},x_{0})\text{,}
\end{eqnarray*}%
a contradiction because of $b\in \lbrack 0,1)$. Consequently, $T$ fixes the
point $x_{0}\in X$, that is, $Tx_{0}=x_{0}$.
\end{proof}

Let the number $r$ be defined as follows:%
\begin{equation}
r=\inf \left\{ d(Tu,u):Tu\neq u,u\in X\right\} \text{.}
\label{definition of r}
\end{equation}

\begin{theorem}
\label{thm1} If $T$ is a Zamfirescu type $x_{0}$-mapping with $x_{0}\in X$
and $d(Tu,x_{0})\leq r$ for each $u\in D(x_{0},r)-\left\{ x_{0}\right\} $,
then $D[x_{0},r]$ is a fixed disc of $T$.
\end{theorem}

\begin{proof}
Suppose that $r=0$. Then we get $D[x_{0},r]=\left\{ x_{0}\right\} $. By
Proposition \ref{prop1}, we have $Tx_{0}=x_{0}$ whence $D[x_{0},r]$ is a
fixed disc of $T$.

Now assume that $r>0$ and $u\in D[x_{0},r]-\left\{ x_{0}\right\} $ is any
point such that $Tu\neq u$. Then we have $d(Tu,u)>0$. Using the Zamfirescu
type $x_{0}$-mapping property, the hypothesis $d(Tu,x_{0})\leq r$ and
Proposition \ref{prop1}, we get%
\begin{eqnarray}
d(Tu,u) &\leq &\max \left\{ ad(u,x_{0}),\frac{b}{2}\left[
d(Tx_{0},u)+d(Tu,x_{0})\right] \right\}   \notag \\
&\leq &\max \left\{ ar,br\right\} \text{.}  \label{eqn2}
\end{eqnarray}%
Without loss of generality we can assume $a\geq b$. Then using the
inequality (\ref{eqn2}), we obtain%
\begin{equation*}
d(Tu,u)\leq ar\text{,}
\end{equation*}%
which is a contradiction with the definition of $r$ because of $a\in \lbrack
0,1)$. Consequently, it should be $Tu=u$ and so $D[x_{0},r]$ is a fixed disc
of $T$.
\end{proof}

From now on, $\Theta $ denotes the class of all increasing functions $\Psi
:[0,1]\rightarrow \lbrack 0,\infty )$ with $\Psi (0)=0$. Modifying the
notion of a Pata type contraction (see \cite{Pata}) and using this class $%
\Theta $, we give the following definition that exclude the continuity
hypothesis on $\Psi $.

\begin{definition}
\label{def2} Let $\Lambda \geq 0$, $\alpha \geq 1$ and $\beta \in \lbrack
0,\alpha ]$ be any constants. Then $T$ is called a Pata type $x_{0}$-mapping
if there exist $x_{0}\in X$ and $\Psi \in \Theta $ such that%
\begin{equation*}
d(Tu,u)>0\Longrightarrow d(Tu,u)\leq \frac{1-\varepsilon }{2}\left\Vert
u\right\Vert +\Lambda \varepsilon ^{\alpha }\Psi (\varepsilon )\left[
1+\left\Vert u\right\Vert +\left\Vert Tu\right\Vert \right] ^{\beta }\text{,}
\end{equation*}%
for all $u\in X$ and each $\varepsilon \in \lbrack 0,1]$, where $\left\Vert
u\right\Vert =d(u,x_{0})$.
\end{definition}

\begin{proposition}
\label{prop2} If $T$ is a Pata type $x_{0}$-mapping with $x_{0}\in X$ then
we have $Tx_{0}=x_{0}$.
\end{proposition}

\begin{proof}
Let $T$ be a Pata type $x_{0}$-mapping with $x_{0}\in X$. Assume that $%
Tx_{0}\neq x_{0}$. Then we have $d(Tx_{0},x_{0})>0$. Using the Pata type $%
x_{0}$-mapping hypothesis, we get%
\begin{eqnarray}
d(Tx_{0},x_{0}) &\leq &\frac{1-\varepsilon }{2}\left\Vert x_{0}\right\Vert
+\Lambda \varepsilon ^{\alpha }\Psi (\varepsilon )\left[ 1+\left\Vert
x_{0}\right\Vert +\left\Vert Tx_{0}\right\Vert \right] ^{\beta }  \notag \\
&=&\frac{1-\varepsilon }{2}d(x_{0},x_{0})+\Lambda \varepsilon ^{\alpha }\Psi
(\varepsilon )\left[ 1+d(x_{0},x_{0})+d(Tx_{0},x_{0})\right] ^{\beta }
\notag \\
&=&\Lambda \varepsilon ^{\alpha }\Psi (\varepsilon )\left[ 1+d(Tx_{0},x_{0})%
\right] ^{\beta }\text{.}  \label{eqn3}
\end{eqnarray}%
For $\varepsilon =0$, using the inequality (\ref{eqn3}), we obtain%
\begin{equation*}
d(Tx_{0},x_{0})\leq 0\text{,}
\end{equation*}%
whence it should be $Tx_{0}=x_{0}$.
\end{proof}

\begin{theorem}
\label{thm2} If $T$ is a Pata type $x_{0}$-mapping with $x_{0}\in X$ then $%
D[x_{0},r]$ is a fixed disc of $T$.
\end{theorem}

\begin{proof}
Suppose that $r=0$. Then we get $D[x_{0},r]=\left\{ x_{0}\right\} $. By
Proposition \ref{prop2}, we have $Tx_{0}=x_{0}$ whence $D[x_{0},r]$ is a
fixed disc of $T$. Now assume that $r>0$ and $u\in D[x_{0},r]-\left\{
x_{0}\right\} $ is any point such that $Tu\neq u$. Then we have $d(Tu,u)>0$.
Using the Pata type $x_{0}$-mapping property, we get%
\begin{equation}
d(Tu,u)\leq \frac{1-\varepsilon }{2}\left\Vert u\right\Vert +\Lambda
\varepsilon ^{\alpha }\Psi (\varepsilon )\left[ 1+\left\Vert u\right\Vert
+\left\Vert Tu\right\Vert \right] ^{\beta }\text{.}  \label{eqn4}
\end{equation}%
For $\varepsilon =0$, using the inequality (\ref{eqn4}), we obtain%
\begin{equation*}
d(Tu,u)\leq \frac{\left\Vert u\right\Vert }{2}=\frac{d(u,x_{0})}{2}\leq
\frac{r}{2}\text{,}
\end{equation*}%
a contradiction with the definition of $r$. Consequently, it should be $Tu=u$%
, that is, $D[x_{0},r]$ is a fixed disc of $T$.
\end{proof}

Combining the notion of a Zamfirescu type $x_{0}$-mapping and a Pata type $%
x_{0}$-mapping, we define the following notion inspiring the concept of a
Pata type Zamfirescu mapping \cite{Jacob}.

\begin{definition}
\label{def3} If there exist $x_{0}\in X$ and $\Psi \in \Theta $ such that%
\begin{equation*}
d(Tu,u)>0\Longrightarrow d(Tu,u)\leq \frac{1-\varepsilon }{2}%
M(u,x_{0})+\Lambda \varepsilon ^{\alpha }\Psi (\varepsilon )\left[
1+\left\Vert u\right\Vert +\left\Vert x_{0}\right\Vert +\left\Vert
Tu\right\Vert +\left\Vert Tx_{0}\right\Vert \right] ^{\beta }\text{,}
\end{equation*}%
for all $u\in X$ and each $\varepsilon \in \lbrack 0,1]$, where $\left\Vert
u\right\Vert =d(u,x_{0})$, $\Lambda \geq 0$, $\alpha \geq 1$, $\beta \in
\lbrack 0,\alpha ]$ are constants and%
\begin{equation*}
M(u,v)=\max \left\{ d(u,v),\frac{d(Tu,u)+d(Tv,v)}{2},\frac{d(Tv,u)+d(Tu,v)}{2%
}\right\} \text{,}
\end{equation*}%
then $T$ is called a Pata Zamfirescu type $x_{0}$-mapping with respect to $%
\Psi \in \Theta $.
\end{definition}

Now we compare a Zamfirescu type $x_{0}$-mapping and a Pata Zamfirescu type $%
x_{0}$-mapping. Let $\gamma =\max \left\{ a,b\right\} $ in Definition \ref%
{def1} and let us consider the Bernoulli's inequality $1+rt\leq (1+t)^{r}$, $%
r\geq 1$ and $t\in \lbrack -1,\infty )$. Then we have%
\begin{eqnarray*}
d(Tu,u) &>&0 \\
&\Longrightarrow &d(Tu,u)\leq \max \left\{ ad(u,x_{0}),\frac{b}{2}\left[
d(Tx_{0},u)+d(Tu,x_{0})\right] \right\} \\
&\leq &\gamma \max \left\{ d(u,x_{0}),\frac{d(Tx_{0},u)+d(Tu,x_{0})}{2}%
\right\} \\
&\leq &\gamma \max \left\{ d(u,x_{0}),\frac{d(Tu,u)+d(Tx_{0},x_{0})}{2},%
\frac{d(Tx_{0},u)+d(Tu,x_{0})}{2}\right\} \\
&\leq &\frac{1-\varepsilon }{2}\max \left\{ d(u,x_{0}),\frac{%
d(Tu,u)+d(Tx_{0},x_{0})}{2},\frac{d(Tx_{0},u)+d(Tu,x_{0})}{2}\right\} \\
&&+\left( \gamma +\frac{\varepsilon -1}{2}\right) \left[ 1+\max \left\{
\left\Vert u\right\Vert +\left\Vert x_{0}\right\Vert ,\frac{\left\Vert
u\right\Vert +\left\Vert x_{0}\right\Vert +\left\Vert Tu\right\Vert
+\left\Vert Tx_{0}\right\Vert }{2}\right\} \right] \\
&\leq &\frac{1-\varepsilon }{2}\max \left\{ d(u,x_{0}),\frac{%
d(Tu,u)+d(Tx_{0},x_{0})}{2},\frac{d(Tx_{0},u)+d(Tu,x_{0})}{2}\right\} \\
&&+\gamma \left( 1+\frac{\varepsilon -1}{\gamma }\right) \left[ 1+\left\Vert
u\right\Vert +\left\Vert x_{0}\right\Vert +\left\Vert Tu\right\Vert
+\left\Vert Tx_{0}\right\Vert \right] \\
&\leq &\frac{1-\varepsilon }{2}M(u,x_{0})+\gamma \varepsilon ^{\frac{1}{%
\gamma }}\left[ 1+\left\Vert u\right\Vert +\left\Vert x_{0}\right\Vert
+\left\Vert Tu\right\Vert +\left\Vert Tx_{0}\right\Vert \right] \\
&\leq &\frac{1-\varepsilon }{2}M(u,x_{0})+\gamma \varepsilon \varepsilon ^{%
\frac{1-\gamma }{\gamma }}\left[ 1+\left\Vert u\right\Vert +\left\Vert
x_{0}\right\Vert +\left\Vert Tu\right\Vert +\left\Vert Tx_{0}\right\Vert %
\right] \text{.}
\end{eqnarray*}%
Consequently, we obtain that a Zamfirescu type $x_{0}$-mapping is a special
case of a Pata Zamfirescu type $x_{0}$-mapping with $\Lambda =\gamma $, $%
\Psi (u)=u^{\frac{1-\gamma }{\gamma }}$ and $\alpha =\beta =1$.

In the following proposition, we see that the point $x_{0}$ in the notion of
a Pata Zamfirescu type $x_{0}$-mapping is a fixed point of a self-mapping $T$%
.

\begin{proposition}
\label{prop3} If $T$ is a Pata Zamfirescu type $x_{0}$-mapping with respect
to $\Psi \in \Theta $ for $x_{0}\in X$ then we have $Tx_{0}=x_{0}$.
\end{proposition}

\begin{proof}
Let $T$ be a Pata Zamfirescu type $x_{0}$-mapping with respect to $\Psi \in
\Theta $ for $x_{0}\in X$. Suppose that $Tx_{0}\neq x_{0}$. Then we have $%
d(Tx_{0},x_{0})>0$. Using the Pata Zamfirescu type $x_{0}$-mapping
hypothesis, we obtain%
\begin{eqnarray}
d(Tx_{0},x_{0}) &\leq &\frac{1-\varepsilon }{2}M(x_{0},x_{0})+\Lambda
\varepsilon ^{\alpha }\Psi (\varepsilon )\left[ 1+2\left\Vert
x_{0}\right\Vert +2\left\Vert Tx_{0}\right\Vert \right] ^{\beta }  \notag \\
&=&\frac{1-\varepsilon }{2}d(Tx_{0},x_{0})+\Lambda \varepsilon ^{\alpha
}\Psi (\varepsilon )\left[ 1+2\left\Vert x_{0}\right\Vert +2\left\Vert
Tx_{0}\right\Vert \right] ^{\beta }\text{.}  \label{eqn5}
\end{eqnarray}%
For $\varepsilon =0$, using the inequality (\ref{eqn5}), we get%
\begin{equation*}
d(Tx_{0},x_{0})\leq \frac{d(Tx_{0},x_{0})}{2}\text{,}
\end{equation*}%
a contradiction. Hence it should be $Tx_{0}=x_{0}$.
\end{proof}

Using Proposition \ref{prop3}, we give the following fixed-disc theorem.

\begin{theorem}
\label{thm3} If $T$ is a Pata Zamfirescu type $x_{0}$-mapping with respect
to $\Psi \in \Theta $ for $x_{0}\in X$ and $d(Tu,x_{0})\leq r$ for each $%
u\in D[x_{0},r]-\left\{ x_{0}\right\} $, then $D[x_{0},r]$ is a fixed disc
of $T$.
\end{theorem}

\begin{proof}
Suppose that $r=0$. Then we get $D[x_{0},r]=\left\{ x_{0}\right\} $. By
Proposition \ref{prop3}, we have $Tx_{0}=x_{0}$ whence $D[x_{0},r]$ is a
fixed disc of $T$. Now assume that $r>0$ and $u\in D[x_{0},r]-\left\{
x_{0}\right\} $ is any point such that $Tu\neq u$. Then we have $d(Tu,u)>0$.
Using the Pata Zamfirescu type $x_{0}$-mapping property, the hypothesis $%
d(Tu,x_{0})\leq r$ and Proposition \ref{prop3}, we obtain%
\begin{eqnarray}
d(Tu,u) &\leq &\frac{1-\varepsilon }{2}M(u,x_{0})+\Lambda \varepsilon
^{\alpha }\Psi (\varepsilon )\left[ 1+\left\Vert u\right\Vert +\left\Vert
x_{0}\right\Vert +\left\Vert Tu\right\Vert +\left\Vert Tx_{0}\right\Vert %
\right] ^{\beta }  \notag \\
&=&\frac{1-\varepsilon }{2}\max \left\{ d(u,x_{0}),\frac{%
d(Tu,u)+d(Tx_{0},x_{0})}{2},\frac{d(Tx_{0},u)+d(Tu,x_{0})}{2}\right\}  \notag
\\
&&+\Lambda \varepsilon ^{\alpha }\Psi (\varepsilon )\left[ 1+\left\Vert
u\right\Vert +\left\Vert x_{0}\right\Vert +\left\Vert Tu\right\Vert
+\left\Vert Tx_{0}\right\Vert \right] ^{\beta }  \notag \\
&\leq &\frac{1-\varepsilon }{2}\max \left\{ r,\frac{d(Tu,u)}{2},r\right\}
\notag \\
&&+\Lambda \varepsilon ^{\alpha }\Psi (\varepsilon )\left[ 1+\left\Vert
u\right\Vert +\left\Vert x_{0}\right\Vert +\left\Vert Tu\right\Vert
+\left\Vert Tx_{0}\right\Vert \right] ^{\beta }\text{.}  \label{eqn6}
\end{eqnarray}%
For $\varepsilon =0$, using the inequality (\ref{eqn6}), we get%
\begin{equation*}
d(Tu,u)\leq \frac{1}{2}\max \left\{ r,\frac{d(Tu,u)}{2}\right\} \text{.}
\end{equation*}%
Hence we get two cases as follows:

\textbf{Case 1:} If $\max \left\{ r,\frac{d(Tu,u)}{2}\right\} =r$ then we
have%
\begin{equation*}
d(Tu,u)\leq \frac{r}{2}\text{,}
\end{equation*}%
a contradiction with the definition of $r$.

\textbf{Case 2:} If $\max \left\{ r,\frac{d(Tu,u)}{2}\right\} =\frac{d(Tu,u)%
}{2}$ then we find%
\begin{equation*}
d(Tu,u)\leq \frac{d(Tu,u)}{2}\text{,}
\end{equation*}%
a contradiction.

Consequently, it should be $Tu=u$ and so $T$ fixes the disc $D[x_{0},r]$.
\end{proof}

We give some illustrative examples to show the validity of our obtained
results.

\begin{example}
\label{exm1} Let $X=%
%TCIMACRO{\U{211d} }%
%BeginExpansion
\mathbb{R}
%EndExpansion
$ be the usual metric space with the metric $d(u,v)=\left\vert
u-v\right\vert $ for all $u,v\in
%TCIMACRO{\U{211d} }%
%BeginExpansion
\mathbb{R}
%EndExpansion
$. Let us define the self-mapping $T:%
%TCIMACRO{\U{211d} }%
%BeginExpansion
\mathbb{R}
%EndExpansion
\rightarrow
%TCIMACRO{\U{211d} }%
%BeginExpansion
\mathbb{R}
%EndExpansion
$ as%
\begin{equation*}
Tu=\left\{
\begin{array}{ccc}
u & \text{if} & u\in \lbrack -4,4] \\
u+1 & \text{if} & u\in (-\infty ,-4)\cup (4,\infty )%
\end{array}%
\right. \text{,}
\end{equation*}%
for all $u\in
%TCIMACRO{\U{211d} }%
%BeginExpansion
\mathbb{R}
%EndExpansion
$. Then

$\bullet $ The self-mapping $T$ is a Zamfirescu type $x_{0}$-mapping with $%
x_{0}=0$, $a=\frac{1}{2}$ and $b=0$. Indeed, we get%
\begin{equation*}
d(Tu,u)=1>0\text{,}
\end{equation*}%
for all $u\in (-\infty ,-4)\cup (4,\infty )$. Hence we find%
\begin{equation*}
d(Tu,u)=1\leq \frac{\left\vert u\right\vert }{2}=\max \left\{ ad(u,0),\frac{b%
}{2}\left[ d(0,u)+d(u+1,0)\right] \right\} \text{.}
\end{equation*}

$\bullet $ The self-mapping $T$ is a Pata type $x_{0}$-mapping with $x_{0}=0$%
, $\Lambda =\alpha =\beta =1$ and
\begin{equation*}
\Psi (u)=\left\{
\begin{array}{ccc}
0 & \text{if} & u=0 \\
\frac{1}{2} & \text{if} & u\in (0,1]%
\end{array}%
\right. .
\end{equation*}%
Indeed, we have%
\begin{equation*}
d(Tu,u)=1>0\text{,}
\end{equation*}%
for all $u\in (-\infty ,-4)\cup (4,\infty )$. So we obtain%
\begin{eqnarray*}
d(Tu,u) &=&1\leq \frac{\left\vert u\right\vert }{2}+\frac{\varepsilon }{2}+%
\frac{\varepsilon \left\vert u+1\right\vert }{2} \\
&=&\frac{1-\varepsilon }{2}\left\Vert u\right\Vert +\Lambda \varepsilon
^{\alpha }\Psi (\varepsilon )\left[ 1+\left\Vert u\right\Vert +\left\Vert
Tu\right\Vert \right] ^{\beta }\text{.}
\end{eqnarray*}

$\bullet $ The self-mapping $T$ is a Pata Zamfirescu type $x_{0}$-mapping
with $x_{0}=0$, $\Lambda =\alpha =\beta =1$ and%
\begin{equation*}
\Psi (u)=\left\{
\begin{array}{ccc}
0 & \text{if} & u=0 \\
\frac{1}{2} & \text{if} & u\in (0,1]%
\end{array}%
\right. .
\end{equation*}%
Indeed, we get%
\begin{equation*}
d(Tu,u)=1>0\text{,}
\end{equation*}%
for all $u\in (-\infty ,-4)\cup (4,\infty )$ and%
\begin{equation*}
M(u,0)=\max \left\{ \left\vert u\right\vert ,\frac{1}{2},\frac{\left\vert
u\right\vert +\left\vert u+1\right\vert }{2}\right\} =\max \left\{
\left\vert u\right\vert ,\frac{\left\vert u\right\vert +\left\vert
u+1\right\vert }{2}\right\} \text{.}
\end{equation*}%
Then we obtain two cases$:$

\textbf{Case }$1:$ Let $\left\vert u\right\vert >\frac{\left\vert
u\right\vert +\left\vert u+1\right\vert }{2}$. We find $M(u,0)=\left\vert
u\right\vert $ and%
\begin{eqnarray*}
d(Tu,u) &=&1\leq \frac{\left\vert u\right\vert }{2}+\frac{\varepsilon }{2}+%
\frac{\varepsilon \left\vert u+1\right\vert }{2} \\
&=&\frac{1-\varepsilon }{2}M(u,0)+\Lambda \varepsilon ^{\alpha }\Psi
(\varepsilon )\left[ 1+\left\Vert u\right\Vert +\left\Vert 0\right\Vert
+\left\Vert Tu\right\Vert +\left\Vert T0\right\Vert \right] ^{\beta }\text{.}
\end{eqnarray*}

\textbf{Case }$2:$ Let $\left\vert u\right\vert <\frac{\left\vert
u\right\vert +\left\vert u+1\right\vert }{2}$. We obtain $M(u,0)=\frac{%
\left\vert u\right\vert +\left\vert u+1\right\vert }{2}$ and%
\begin{eqnarray*}
d(Tu,u) &=&1\leq \frac{\left\vert u\right\vert }{4}+\frac{\varepsilon
\left\vert u\right\vert }{2}+\frac{\left\vert u+1\right\vert }{4}+\frac{%
\varepsilon \left\vert u+1\right\vert }{2}+\frac{\varepsilon }{2} \\
&=&\frac{1-\varepsilon }{2}M(u,0)+\Lambda \varepsilon ^{\alpha }\Psi
(\varepsilon )\left[ 1+\left\Vert u\right\Vert +\left\Vert 0\right\Vert
+\left\Vert Tu\right\Vert +\left\Vert T0\right\Vert \right] ^{\beta }\text{.}
\end{eqnarray*}

Also, we find%
\begin{equation*}
r=\underset{u\in X}{\inf }\left\{ d(Tu,u):Tu\neq u\right\} =1
\end{equation*}%
and%
\begin{equation*}
d(Tu,0)=d(u,0)\leq 1\text{,}
\end{equation*}%
for all $u\in D[0,1]-\{0\}$. Consequently, from Theorem \ref{thm1} $($resp.
Theorem \ref{thm2} and Theorem \ref{thm3}$)$, $T$ fixes the disc $D[0,1]$.
\end{example}

\begin{example}
\label{exm2} Let $X=[0,1]$ be the usual metric space. Let us define the
self-mapping $T:X\rightarrow X$ as%
\begin{equation*}
Tu=\left\{
\begin{array}{ccc}
u & \text{if} & u\in \{0,1\} \\
2u & \text{if} & u\in (0,1)%
\end{array}%
\right. \text{,}
\end{equation*}%
for all $u\in X$. Then $T$ is a Zamfirescu type $x_{0}$-mapping with $%
x_{0}=0 $, $a=0$ and $b=\frac{2}{3}$, but $T$ is not a Zamfirescu type $%
x_{0} $-mapping with $x_{0}=1$. Also we get $r=0$ and so $T$ fixes the point
$x_{0}=0$.
\end{example}

\begin{example}
\label{exm3} Let $X=%
%TCIMACRO{\U{211d} }%
%BeginExpansion
\mathbb{R}
%EndExpansion
$ be the usual metric space. Let us define the self-mapping $T:%
%TCIMACRO{\U{211d} }%
%BeginExpansion
\mathbb{R}
%EndExpansion
\rightarrow
%TCIMACRO{\U{211d} }%
%BeginExpansion
\mathbb{R}
%EndExpansion
$ as%
\begin{equation*}
Tu=\left\{
\begin{array}{ccc}
u & \text{if} & u\in \lbrack -2,\infty ) \\
u+1 & \text{if} & u\in (-\infty ,-2)%
\end{array}%
\right. \text{,}
\end{equation*}%
for all $u\in
%TCIMACRO{\U{211d} }%
%BeginExpansion
\mathbb{R}
%EndExpansion
$. Then $T$ is a Zamfirescu type $x_{0}$-mapping with $a=\frac{1}{2}$, $b=0$%
, both $x_{0}=0$ and $x_{0}=5$. We obtain $r=1$ whence by Theorem \ref{thm1}
$T$ fixes both of the discs $D[0,1]$ and $D[5,1]$.
\end{example}

\begin{example}
\label{exm4} Let $X=%
%TCIMACRO{\U{211d} }%
%BeginExpansion
\mathbb{R}
%EndExpansion
$ be the usual metric space. Let us define the self-mapping $T:%
%TCIMACRO{\U{211d} }%
%BeginExpansion
\mathbb{R}
%EndExpansion
\rightarrow
%TCIMACRO{\U{211d} }%
%BeginExpansion
\mathbb{R}
%EndExpansion
$ as%
\begin{equation*}
Tu=\left\{
\begin{array}{ccc}
u & \text{if} & u\in \lbrack -1,1] \\
0 & \text{if} & u\in (-\infty ,-1)\cup (1,\infty )%
\end{array}%
\right. \text{,}
\end{equation*}%
for all $u\in
%TCIMACRO{\U{211d} }%
%BeginExpansion
\mathbb{R}
%EndExpansion
$. Then we have $r=1$ and $T$ is not a Zamfirescu type $x_{0}$-mapping $($%
resp. a Pata type $x_{0}$-mapping and a Pata Zamfirescu type $x_{0}$-mapping$%
)$ with any $x_{0}\in X$ but $T$ fixes the disc $D[0,1]$.
\end{example}

Considering the above examples, we conclude the following remarks.

\begin{remark}
\label{rem1} $(1)$ The point $x_{0}$ satisfying the definition of a
Zamfirescu type $x_{0}$-mapping $($resp. a Pata type $x_{0}$-mapping and a
Pata Zamfirescu type $x_{0}$-mapping$)$ is a fixed point of the self-mapping
$T$. But the converse statement is not always true, that is, a fixed point
of $T$ does not always satisfy the definition of a Zamfirescu type $x_{0}$%
-mapping $($resp. a Pata type $x_{0}$-mapping and a Pata Zamfirescu type $%
x_{0}$-mapping$)$. For example, if we consider Example \ref{exm2} then $T$
fixes the point $x_{0}=1$, but the point $1$ does not satisfy the definition
of a Zamfirescu type $x_{0}$-mapping.

$(2)$ The choice of $x_{0}$ is independent from the number $r$ $($see
Example \ref{exm1}, Example \ref{exm2} and Example \ref{exm3}$)$.

$(3)$ The radius $r$ can be zero $($see Example \ref{exm2}$)$.

$(4)$ The number of $x_{0}$ satisfying the definition of a Zamfirescu type $%
x_{0}$-mapping $($resp. a Pata type $x_{0}$-mapping and a Pata Zamfirescu
type $x_{0}$-mapping$)$ can be more than one $($see Example \ref{exm3}$)$.

$(5)$ The converse statements of Theorem \ref{thm1}, Theorem \ref{thm2} and
Theorem \ref{thm3} is not always true $($see Example \ref{exm4}$)$.

$(6)$ The obtained fixed-disc results can be also considered as fixed-circle
results $($resp. fixed-point results$)$.
\end{remark}

\section{A Best Proximity Circle Application}

\label{sec:2} In this section, we define the notion of a best proximity
circle on a metric space. At first, we recall the definition of a best
proximity point and some basic concepts. Let $A$, $B$ be two nonempty
subsets of a metric space $(X,d)$. We consider the followings:

\begin{equation*}
\mathsf{d}\left( A,B\right) =\inf \left\{ d(u,v):u\in A\text{ and }v\in
B\right\} ,
\end{equation*}%
\begin{equation*}
A_{0}=\left\{ u\in A:d\left( u,v\right) =\mathsf{d}\left( A,B\right) \text{
for some }v\in B\right\}
\end{equation*}%
and%
\begin{equation*}
B_{0}=\left\{ v\in B:d\left( u,v\right) =\mathsf{d}\left( A,B\right) \text{
for some }u\in A\right\} .
\end{equation*}%
For a mapping $T:A\rightarrow B$, the point $u\in A$ is called a best
proximity point of $T$ if
\begin{equation*}
d\left( u,Tu\right) =\mathsf{d}\left( A,B\right) .
\end{equation*}%
If $T$ has more than one best proximity point then it is an interesting
problem to consider the geometric properties of these points. For this
purpose we define a circle $C_{x_{0},r}=\left\{ u\in A:d\left(
u,x_{0}\right) =r\right\} $ as the best proximity circle of $T$ if
\begin{equation*}
d\left( u,Tu\right) =\mathsf{d}\left( A,B\right) ,
\end{equation*}%
for all $u\in C_{x_{0},r}$. We note that the best proximity circle reduces
to a fixed circle of $T$ (the circle $C_{x_{0},r}$ is called as the fixed
circle of $T$ if $Tu=u$ for every $u\in C_{x_{0},r}$ \cite%
{Ozgur-Tas-malaysian}). Also if $C_{x_{0},r}$ has only one element then the
best proximity circle reduces to a best proximity point or a fixed point of $%
T$. Using this notion we give an application to a Pata type $x_{0}$-mapping.

\begin{definition}
\label{def4} Let $(X,d)$ be a metric space and $T:A\rightarrow B$ be a
mapping. Then $T$ is called a Pata type proximal $x_{0}$-contraction if
there exists $x_{0}\in A_{0}$ and $\Psi \in \Theta $ such that%
\begin{equation*}
d(x,u)\leq \frac{1-\varepsilon }{2}d(x,x_{0})+\Lambda \varepsilon ^{\alpha
}\Psi (\varepsilon )\left[ 1+\left\Vert x\right\Vert +\left\Vert
x_{0}\right\Vert \right] ^{\beta }\text{,}
\end{equation*}%
for all $x\in A$ and each $\varepsilon \in \lbrack 0,1]$, where $d(u,Tx)=%
\mathsf{d}(A,B)$, $\left\Vert x\right\Vert =d(x,x_{0})$ and $\Lambda \geq 0$%
, $\alpha \geq 1$, $\beta \in \lbrack 0,\alpha ]$ are any constants.
\end{definition}

\begin{proposition}
\label{prop4} If $T$ is a Pata type proximal $x_{0}$-contraction with $%
x_{0}\in A_{0}$ such that $TA_{0}\subset B_{0}$, then $x_{0}$ is a best
proximity point of $T$ in $A$.
\end{proposition}

\begin{proof}
Let $x_{0}\in A_{0}$. Then we have $Tx_{0}\in B_{0}$ because of $%
TA_{0}\subset B_{0}$. Hence there exists $u\in A$ such that
\begin{equation*}
d(u,Tx_{0})=\mathsf{d}(A,B)\text{.}
\end{equation*}%
Using the Pata type proximal $x_{0}$-contractive property, we obtain%
\begin{equation}
d(x_{0},u)\leq \frac{1-\varepsilon }{2}d(x_{0},x_{0})+\Lambda \varepsilon
^{\alpha }\Psi (\varepsilon )\left[ 1+\left\Vert x_{0}\right\Vert
+\left\Vert x_{0}\right\Vert \right] ^{\beta }\text{.}  \label{eqn7}
\end{equation}%
For $\varepsilon =0$, using the inequality (\ref{eqn7}), we get%
\begin{equation*}
d(x_{0},u)\leq 0\text{,}
\end{equation*}%
which implies $x_{0}=u$. Consequently, $x_{0}$ is a best proximity point of $%
T$ in $A$, that is, $d(x_{0},Tx_{0})=\mathsf{d}(A,B)$.
\end{proof}

\begin{theorem}
\label{thm4} Let $\mu =\inf \left\{ d(x,u):x,u\in A\text{ such that }x\neq
u\right\} $. If $T$ is a Pata type proximal $x_{0}$-contraction with $%
x_{0}\in A_{0}$ such that $TA_{0}\subset B_{0}$ and $C_{x_{0},\mu }\subset
A_{0}$, then $C_{x_{0},\mu }$ is a proximity circle of $T$.
\end{theorem}

\begin{proof}
Let $\mu =0$. Then we have $C_{x_{0},\mu }=\{x_{0}\}$. From Proposition \ref%
{prop4}, $C_{x_{0},\mu }$ is a proximity circle of $T$. Now suppose that $%
\mu >0$. Let $x\in C_{x_{0},\mu }-\{x_{0}\}$. Then using the hypothesis $%
C_{x_{0},\mu }\subset A_{0}$, we have $x\in A_{0}$ and so $Tx\in B_{0}$.
Hence there exists $u\in A$ such that%
\begin{equation*}
d(u,Tx)=\mathsf{d}(A,B)\text{.}
\end{equation*}%
Using the Pata type proximal $x_{0}$-contractive property, we obtain%
\begin{equation}
d(x,u)\leq \frac{1-\varepsilon }{2}d(x,x_{0})+\Lambda \varepsilon ^{\alpha
}\Psi (\varepsilon )\left[ 1+\left\Vert x\right\Vert +\left\Vert
x_{0}\right\Vert \right] ^{\beta }\text{.}  \label{eqn8}
\end{equation}%
For $\varepsilon =0$, using the inequality (\ref{eqn8}), we get%
\begin{equation*}
d(x,u)\leq \frac{1}{2}d(x,x_{0})=\frac{\mu }{2}\leq \frac{d(x,u)}{2}\text{,}
\end{equation*}%
which is a contradiction. It should be $x=u$. Consequently, $C_{x_{0},\mu }$
is a proximity circle of $T$.
\end{proof}

\begin{corollary}
\label{cor1} Let $(X,d)$ be a metric space and $T:X\rightarrow X$ be a
mapping which satisfies the conditions of Definition \ref{def2}. Then $T$
has a fixed circle $C_{x_{0},\mu }$ in $X$.
\end{corollary}

\begin{proof}
The proof can be easily obtained from Theorem \ref{thm4} when $A=B=X$.
\end{proof}

Notice that Corollary \ref{cor1} is a special case of Theorem \ref{thm2},
that is, each of the fixed-disc results given in Theorem \ref{thm2} can be
considered as a fixed-circle theorem.

Finally, we give an example to Theorem \ref{thm4} following \cite{Jacob}.

\begin{example}
\label{exm5} Let $A=\left\{ (0,a):a\in \lbrack 0,1]\right\} $ and $B=\left\{
(1,b):b\in \lbrack 0,1]\right\} $ on $%
%TCIMACRO{\U{211d} }%
%BeginExpansion
\mathbb{R}
%EndExpansion
^{2}$ with the metric $d:X\times X\rightarrow
%TCIMACRO{\U{211d} }%
%BeginExpansion
\mathbb{R}
%EndExpansion
$ defined as $d(x,y)=\left\vert x_{1}-y_{1}\right\vert +\left\vert
x_{2}-y_{2}\right\vert $ for all $x=(x_{1},x_{2}),$ $y=(y_{1},y_{2})\in
%TCIMACRO{\U{211d} }%
%BeginExpansion
\mathbb{R}
%EndExpansion
^{2}$. Let us define the mapping $T:A\rightarrow B$ as%
\begin{equation*}
T(0,x)=\left( 1,\frac{2x}{3}\right) \text{,}
\end{equation*}%
for all $(0,x)\in A$. Then $T$ is a Pata type proximal $x_{0}$-contraction
with $x_{0}=(0,0)$, $\Lambda =\alpha =\beta =1$ and%
\begin{equation*}
\Psi (u)=\left\{
\begin{array}{ccc}
0 & \text{if} & u=0 \\
\frac{1}{2} & \text{if} & u\in (0,1]%
\end{array}%
\right. .
\end{equation*}%
Indeed, we get $\mathsf{d}(A,B)=1$ and $\mu =0$. Now we show that $T$
satisfies the Pata type proximal $x_{0}$-contractive property for all $%
\varepsilon \in \lbrack 0,1]$.

Let $\varepsilon =0$. For all $(0,x)=x^{\prime }\in A$, we get%
\begin{equation*}
d(x^{\prime },u)=d\left( \left( 0,x\right) ,\left( 0,\frac{2x}{3}\right)
\right) =\frac{x}{3}\leq \frac{x}{2}=\frac{1}{2}d(x^{\prime },x_{0})\text{,}
\end{equation*}%
where $d(u,Tx)=\mathsf{d}(A,B)=1$.

Let $\varepsilon \in (0,1]$. For all $(0,x)=x^{\prime }\in A$, we get%
\begin{eqnarray*}
d(x^{\prime },u) &=&\frac{x}{3} \\
&\leq &\frac{x}{2}+\frac{\varepsilon }{2} \\
&=&\frac{1-\varepsilon }{2}d(x,x_{0})+\Lambda \varepsilon ^{\alpha }\Psi
(\varepsilon )\left[ 1+\left\Vert x\right\Vert +\left\Vert x_{0}\right\Vert %
\right] ^{\beta }\text{,}
\end{eqnarray*}%
where $d(u,Tx)=\mathsf{d}(A,B)=1$.

Hence $T$ is a Pata type proximal $x_{0}$-contraction and there exists a
best proximity circle $C_{x_{0},\mu }=\{(0,0)\}$ in $A$. Also the circle $%
C_{x_{0},\mu }$ can be considered as a best proximity point.
\end{example}

\end{document}